\documentclass[11pt,a4paper]{article}

\usepackage[a4paper,margin=1in]{geometry}
\usepackage{amsmath,amssymb,amsthm,mathtools}
\usepackage{microtype}
\usepackage{hyperref}
\hypersetup{hidelinks,pdftitle={One-sided certification of surrogate near-minimizers: exact worst-case envelopes under strong convexity and smoothness},pdfauthor={Karim Bounja, Boujemaa Achchab, Abdeljalil Sakat}}

\newtheorem{theorem}{Theorem}[section]
\newtheorem{proposition}[theorem]{Proposition}
\newtheorem{lemma}[theorem]{Lemma}
\newtheorem{corollary}[theorem]{Corollary}
\theoremstyle{remark}
\newtheorem{remark}[theorem]{Remark}

\newcommand{\A}{\mathcal A}

\newcommand{\R}{\mathbb R}

\title{\bfseries One-sided certification of surrogate near-minimizers:\\
exact worst-case envelopes under strong convexity and smoothness}
\author{Karim Bounja$^{*}$ \quad Boujem\^aa Achchab \quad Abdeljalil Sakat\\[0.7em]
\small Laboratory for Analysis and Modeling of Systems and Decision Support (LAMSAD),\\
\small Hassan 1st University of Settat, Morocco\\[0.25em]
\small \texttt{k.bounja.doc@uhp.ac.ma} \quad \texttt{boujemaa.achchab@uhp.ac.ma}\\
\small \texttt{abdeljalil.sakat@uhp.ac.ma}\\[0.2em]
\small $^{*}$Corresponding author}
\date{}

\begin{document}
\emergencystretch=2em
\maketitle

\begin{abstract}
We study target suboptimality for points that approximately minimize a
surrogate objective. Unlike standard uniform perturbation estimates, the
available comparison is one-sided, thresholded at a certified lower target
level, and restricted to the entire surrogate-selected sublevel set. We first
show that selected comparison alone cannot yield a finite worst-case guarantee,
even for strongly convex quadratics with identical Hessians. A single scalar
upper bound on the surrogate minimum closes this information gap and leads to
a sharp general estimate. For pairs of $\mu$-strongly convex and $L$-smooth
objectives, we then solve the associated worst-case information problem exactly.
We obtain a sharp feasibility barrier and a dimension-free envelope; in its
nontrivial regime, the value reduces to a scalar maximization over a compact
interval. Explicit one-dimensional piecewise-quadratic objectives attain the
bound while satisfying the one-sided constraint on the full selected set.
We characterize when curvature strictly improves the general certificate and
provide a threshold-preserving finite-net validation procedure. The results
identify the best guarantee permitted by the prescribed partial information,
independently of any particular optimization algorithm or additive perturbation
model.
\end{abstract}



\section{Introduction}

Let $F$ be a target objective and $\widetilde F$ a surrogate returning an
$\varepsilon$-near-minimizer $\widehat x$, so that
$\widetilde F(\widehat x)\le\inf\widetilde F+\varepsilon$.
We seek a guarantee on $F(\widehat x)-F^\star$ from information available
\emph{after selection}. Uniform two-sided approximation estimates are
sufficient but constrain regions the surrogate cannot select. Here the
comparison is one-sided and restricted to the \emph{entire} surrogate
$\varepsilon$-sublevel set.

Write $m=\inf\widetilde F$ and let $\ell\le F^\star$ be a certified
lower target level. The selected discrepancy
$\delta_{+,\varepsilon}$ compares $F$ with
$T_\ell(\widetilde F)=\max\{\widetilde F,\ell\}$ only where
$\widetilde F\le m+\varepsilon$. This threshold avoids penalizing
surrogate values below the certified target level. Selected comparison
alone is nevertheless insufficient: strongly convex quadratics with the
same Hessian can have $\delta_{+,0}=0$ and arbitrarily large target loss
(Proposition~\ref{prop:impossibility}). One scalar upper control on $m$
closes the gap. Indeed, with $g=F^\star-\ell$, the general sharp bound is
\begin{equation}\label{eq:intro-master}
F(\widehat x)-F^\star
\le\delta_{+,\varepsilon}+(m+\varepsilon-\ell)_+-g.
\end{equation}
Thus the selected discrepancy and the surrogate minimum encode distinct
parts of the necessary information.

Our principal result solves the corresponding \emph{exact worst-case
problem} when both objectives are $\mu$-strongly convex and $L$-smooth.
In the crossing regime $m=\ell-c\le\ell$, and for a prescribed budget
$\delta$, put $d=\delta-g$ and $\rho=\mu/L$.
We obtain a sharp feasibility barrier
$d\ge\rho\min\{c,\varepsilon\}$ (Theorem~\ref{thm:barrier}) and the
attained, dimension-independent supremum of target suboptimality
(Theorem~\ref{thm:envelope}). In the nontrivial regime
$0<c<\varepsilon$ and $\rho c\le d<c$, this envelope reduces to
maximization of an explicit scalar function over a compact interval.
One-dimensional piecewise-quadratic pairs attain the bound while
respecting the comparison on the full selected sublevel set.
Corollary~\ref{cor:phase} determines exactly when curvature yields a
strict improvement over \eqref{eq:intro-master}.

\paragraph{Relation to previous work.}
Quantitative stability of approximate minimizers under objective and
epigraphical perturbations is classical
\cite{AttouchWets91,AttouchWets93,RW,BonnansShapiro1998,Royset2018,RoysetWets2019}.
Selected-point surrogate majorization appears in
\cite{Mairal2013}; algorithmic inexact-oracle certificates and smooth
performance estimation address different information models
\cite{DevolderGlineurNesterov2014,GladinGasnikovDvurechensky2025,DroriTeboulle2014}.
Smooth strongly convex interpolation provides the geometric inequalities
used here \cite{TaylorHendrickxGlineur2017}. The distinguishing step is
an exact reduction under a \emph{continuum} of constraints coupling two
objectives, followed by globally admissible extremizers. The framework
also differs from structured smooth perturbation estimates
\cite{Spokoiny2026} and statistical localization bounds
\cite{XuZeevi2025}.

The general information bound is developed in Section~\ref{sec:information};
Sections~\ref{sec:geometry}--\ref{sec:envelope} establish the exact
envelope, followed by the strict-improvement criterion and finite
validation in Sections~\ref{sec:phase}--\ref{sec:validation}.

\section{Selected one-sided information}\label{sec:information}

Let $B$ be a nonempty decision region and let
$F,\widetilde F:B\to\R$ be real-valued objectives with
\[
F^\star:=\inf_B F> -\infty,
\qquad
m:=\inf_B\widetilde F> -\infty.
\]
Fix a certified lower target level $\ell\le F^\star$ and define
\[
T_\ell(r):=\max\{r,\ell\},
\qquad
g:=F^\star-\ell\ge0.
\]
For $\varepsilon\ge0$, the surrogate-selected set is
\[
\A_\varepsilon(\widetilde F)
:=\{x\in B:\widetilde F(x)\le m+\varepsilon\}.
\]
Whenever $\A_\varepsilon(\widetilde F)$ is nonempty, the retained
forward discrepancy is
\begin{equation}
\label{eq:delta}
\delta_{+,\varepsilon}
:=
\sup_{x\in\A_\varepsilon(\widetilde F)}
\bigl[F(x)-T_\ell(\widetilde F(x))\bigr]_+.
\end{equation}
Since $F\ge F^\star\ge\ell$, thresholding is needed only on the surrogate
side.

\subsection{Forward information alone is insufficient}

\begin{proposition}[Forward-only impossibility]
\label{prop:impossibility}
Fix $\ell\in\R$. For every $M\ge1$ there exist $C^\infty$ strongly convex
quadratics $F_M,\widetilde F_M$ on $[0,1]$, both with Hessian $2$, such that
\[
F_M^\star=\ell,
\qquad
\A_0(\widetilde F_M)=\{1\},
\qquad
\delta_{+,0}=0,
\]
while
\[
F_M(1)-F_M^\star=M.
\]
Thus selected forward information alone gives no finite uniform target-loss
bound, even under common smoothness and strong-convexity moduli.
\end{proposition}

\begin{proof}
Take
\[
F_M(x)=\ell+(M-1)x+x^2,
\qquad
\widetilde F_M(x)=\ell+M+(x-1)^2.
\]
Their minimizers on $[0,1]$ are $0$ and $1$, respectively. At the unique
selected point,
\[
F_M(1)=\widetilde F_M(1)=\ell+M,
\]
so $\delta_{+,0}=0$, whereas $F_M(1)-F_M^\star=M$.
\end{proof}

The uncontrolled quantity in this example is the vertical surrogate level
$m=\inf\widetilde F$.

\subsection{One scalar level certificate}

\begin{theorem}[Selected-set stability]
\label{thm:master}
For every $\widehat x\in\A_\varepsilon(\widetilde F)$,
\begin{equation}
\label{eq:master}
F(\widehat x)-F^\star
\le
\delta_{+,\varepsilon}
+(m+\varepsilon-\ell)_+-g.
\end{equation}
\end{theorem}

\begin{proof}
By \eqref{eq:delta} and the definition of the selected set,
\[
F(\widehat x)
\le T_\ell(\widetilde F(\widehat x))+\delta_{+,\varepsilon}
\le T_\ell(m+\varepsilon)+\delta_{+,\varepsilon}.
\]
Subtract $F^\star=\ell+g$.
\end{proof}

Suppose now that the surrogate minimum crosses the lower target level,
\[
m=\ell-c,
\qquad c\ge0.
\]
Then \eqref{eq:master} becomes
\begin{equation}
\label{eq:crossing}
F(\widehat x)-F^\star
\le
\delta_{+,\varepsilon}+(
\varepsilon-c)_+-g.
\end{equation}
Thus the exact surrogate level disappears from the bound after crossing; for
$c\ge\varepsilon$ there is no residual level penalty.

The elementary estimate is nevertheless best possible under the retained
information.

\begin{proposition}[Unrestricted sharpness]
\label{prop:unrestrictedsharp}
Fix $\ell,m\in\R$ and $g,\varepsilon,\delta\ge0$ satisfying
\[
\delta\ge\bigl[g-(m-\ell)_+\bigr]_+.
\]
There exist convex Lipschitz functions $F,\widetilde F$ on $[0,1]$ and
$\widehat x\in\A_\varepsilon(\widetilde F)$ for which
\[
F^\star=\ell+g,
\qquad
\inf\widetilde F=m,
\qquad
\delta_{+,\varepsilon}=\delta,
\]
and equality holds in \eqref{eq:master}.
\end{proposition}

\begin{proof}
Set
\[
\widetilde F(x)=m+\varepsilon x,
\qquad
H(x)=T_\ell(\widetilde F(x)),
\]
and
\[
F(x)=\ell+g+x\bigl(H(x)+\delta-(\ell+g)\bigr).
\]
The feasibility condition gives $H(0)+\delta\ge\ell+g$, hence
$F\ge F(0)=\ell+g$. Moreover
\[
F(x)-(H(x)+\delta)
=(1-x)(\ell+g-H(x)-\delta)\le0,
\]
with equality at $x=1$. Since
$\A_\varepsilon(\widetilde F)=[0,1]$, the discrepancy equals $\delta$ and
$\widehat x=1$ attains equality in \eqref{eq:master}. Convexity follows from
\[
F(x)=\ell+g+(\delta-g)x+x(m-\ell+\varepsilon x)_+.
\]
Indeed, $x\mapsto x(b+\varepsilon x)_+$ is convex for $x\ge0$:
it vanishes where $b+\varepsilon x\le0$ and is quadratic with
nondecreasing derivative on the remaining interval (with an upward
slope jump at a positive crossing).
\end{proof}

\section{Fixed-curvature geometry}\label{sec:geometry}

The preceding section identifies the information needed for finite-level
stability. Curvature is not required for that closure. We now keep exactly the
same reduced information and ask how fixed geometry further restricts the
worst case. For $0<\mu\le L$, let $\mathcal F_{\mu,L}(\R^n)$ be the class of
continuously differentiable functions that are $\mu$-strongly convex and have
$L$-Lipschitz gradients.

\begin{theorem}[Sharp curvature barrier]
\label{thm:barrier}
Let
\[
F,\widetilde F\in\mathcal F_{\mu,L}(\R^n),
\qquad
F^\star=\ell+g,
\qquad
\inf\widetilde F=\ell-c,
\]
with $c\ge0$. Then
\begin{equation}
\label{eq:barrier}
\delta_{+,\varepsilon}-g
\ge
\frac{\mu}{L}\min\{c,\varepsilon\}.
\end{equation}
The constant $\mu/L$ is sharp.
\end{theorem}

\begin{proof}
Let $\widetilde x^\star$ minimize $\widetilde F$ and put
\[
q:=\min\{c,\varepsilon\},
\qquad
r:=\sqrt{\frac{2q}{L}}.
\]
By $L$-smoothness,
\[
\widetilde F(x)
\le
\ell-c+\frac L2\|x-\widetilde x^\star\|^2,
\]
so
\[
\overline B(\widetilde x^\star,r)
\subset
\A_\varepsilon(\widetilde F)\cap\{\widetilde F\le\ell\}.
\]
Hence on this ball
\[
F(x)-F^\star\le\delta_{+,\varepsilon}-g.
\]
If $x^\star$ minimizes $F$, strong convexity gives
\[
F(x)-F^\star\ge\frac\mu2\|x-x^\star\|^2.
\]
The closed ball of radius $r$ contains a point at distance at least $r$ from
$x^\star$, proving \eqref{eq:barrier}.

For sharpness, in dimension one take
\[
F(x)=\ell+g+\frac\mu2x^2,
\qquad
\widetilde F(x)=\ell-c+\frac L2x^2.
\]
Writing $y=Lx^2/2\in[0,\varepsilon]$ on the selected set gives
\[
F(x)-T_\ell(\widetilde F(x))
=g+\frac\mu L y-(y-c)_+.
\]
Its maximum is attained at $y=\min\{c,\varepsilon\}$ and equals the
right-hand side of \eqref{eq:barrier} plus $g$.
\end{proof}

Thus fixed conditioning does more than improve an estimate: it rules out
certificate values below a sharp geometric floor. We now determine the exact
worst-case target loss above that floor.

\section{Exact fixed-curvature envelope}
\label{sec:envelope}

Fix crossing data $m=\ell-c\le\ell$ and define
\[
\mathfrak E_{\mu,L}(\ell,g,c,\varepsilon,\delta)
:=
\sup\{F(\widehat x)-F^\star\},
\]
where the supremum is over all dimensions $n\ge1$, all
$F,\widetilde F\in\mathcal F_{\mu,L}(\R^n)$, and all
$\widehat x\in\A_\varepsilon(\widetilde F)$ satisfying
\[
F^\star=\ell+g,
\qquad
\inf\widetilde F=\ell-c,
\qquad
\delta_{+,\varepsilon}\le\delta.
\]
Write
\[
\rho:=\frac\mu L,
\qquad
d:=\delta-g.
\]
By Theorem~\ref{thm:barrier}, feasibility requires
\begin{equation}
\label{eq:feascurv}
d\ge\rho\min\{c,\varepsilon\}.
\end{equation}

The nontrivial regime is
\begin{equation}
\label{eq:intermediate}
0<c<\varepsilon,
\qquad
\rho c\le d<c.
\end{equation}
Two endpoint interpolation lemmas are needed.

\begin{lemma}[Minimum surrogate exit slope]
\label{lem:beta}
Assume $0<\mu<L$, $c>0$, and set
\[
r:=\sqrt{\frac{2c}{L}},
\qquad
R_c:=\sqrt{\frac{2c}{\mu}}.
\]
If $h\in\mathcal F_{\mu,L}(\R)$ satisfies
$h(0)=h'(0)=0$ and $h(a)=c$ for $a\in[r,R_c]$, then
\begin{equation}
\label{eq:beta}
h'(a)\ge
\beta(a):=
La-\sqrt{(L-\mu)(La^2-2c)}.
\end{equation}
Equality is attained by a piecewise-quadratic $h$ with curvature $L$ first
and $\mu$ afterwards. Moreover
\begin{equation}
\label{eq:beta-range}
\sqrt{2\mu c}\le\beta(a)\le\sqrt{2Lc}.
\end{equation}
\end{lemma}

\begin{proof}
The two-point smooth strongly convex interpolation inequality
\cite{TaylorHendrickxGlineur2017}, applied to $(0,0,0)$ and
$(a,h'(a),c)$ in reverse order, gives
\[
 ah'(a)-c\ge \frac{\mu a^2}{2}
     +\frac{(h'(a)-\mu a)^2}{2(L-\mu)}.
\]
Rearrangement yields
$\bigl(h'(a)-La\bigr)^2\le(L-\mu)(La^2-2c)$, and hence
\eqref{eq:beta}. Put
$s=a-\sqrt{(La^2-2c)/(L-\mu)}\in[0,a]$ and define
\begin{equation}\label{eq:beta-extremizer}
 h_a(t):=\frac L2t^2-\frac{L-\mu}{2}(t-s)_+^2
 \qquad(t\in\R).
\end{equation}
Then $h_a(0)=h'_a(0)=0$, $h_a(a)=c$, and
$h'_a(a)=\beta(a)$. Its a.e. second derivative belongs to
$\{\mu,L\}$, proving sharpness in $\mathcal F_{\mu,L}$.
Finally, the standard gradient inequalities imply
$2\mu c\le |h'(a)|^2\le2Lc$, giving \eqref{eq:beta-range}.
\end{proof}

\begin{lemma}[Maximum target exit slope under a level cap]
\label{lem:Ad}
Let $d>0$ and set
\[
R_d:=\sqrt{\frac{2d}{\mu}},
\qquad
S_d:=\sqrt{\frac{2d}{L}}.
\]
Suppose $F\in\mathcal F_{\mu,L}(\R^n)$ and
$F-F^\star\le d$ on a line segment of length $0<W\le2R_d$. At the right
endpoint, the directional derivative along the segment is at most
\begin{equation}
\label{eq:Ad}
A_d(W):=
\begin{cases}
\sqrt{2Ld},
& W\le R_d+S_d,\\[2mm]
\displaystyle
\mu(W-R_d)+
\sqrt{2(L-\mu)
\left[d-\dfrac\mu2(W-R_d)^2\right]},
& \begin{aligned}R_d+S_d<W\\[-1mm]{}\le2R_d.
\end{aligned}
\end{cases}
\end{equation}
The bound is sharp in dimension one. In addition, every terminal slope in
$[\sqrt{2\mu d},A_d(W)]$ is realized by a one-dimensional member of
$\mathcal F_{\mu,L}$ with terminal value $d$ and value at most $d$ on the
whole segment.
\end{lemma}

\begin{proof}
For $\mu=L$, the bound and its attainment follow from the quadratic
$\varphi(t)=\mu(t-W+R_d)^2/2$; thus assume $\mu<L$.
Restrict $F-F^\star$ to the supporting line, subtract its nonnegative
minimum, and denote the resulting one-dimensional function by
$\varphi$. It has minimum zero and $\varphi\le d$ on $[0,W]$.
If $\varphi'(W)\le0$, the upper bound is immediate. Otherwise its
minimizer $t_\star$ lies to the left of $W$; write $z=W-t_\star$.
Strong convexity at both endpoints gives
\[
 \max\{0,W-R_d\}\le z\le R_d.
\]
The interpolation inequality between $t_\star$ and $W$ yields
\[
 d\ge\varphi(W)\ge\frac{\mu z^2}{2}
 +\frac{(\varphi'(W)-\mu z)^2}{2(L-\mu)}.
\]
Consequently,
\[
 \varphi'(W)\le \mu z+
 \sqrt{(L-\mu)(2d-\mu z^2)}.
\]
This expression increases up to $z=S_d=\sqrt{2d/L}$ and decreases
thereafter. Maximization on the admissible interval gives exactly the
two branches of \eqref{eq:Ad}.

We now construct extremizers and realize the full interval of
terminal slopes. Put
\[
 z_0:=\max\{S_d,W-R_d\}\in[S_d,R_d],
\]
\[
 v(z):=\sqrt{\frac{2d-\mu z^2}{L-\mu}}
 \qquad(z\in[z_0,R_d]),
\]
and, on the entire real line, define
\begin{equation}\label{eq:Ad-extremizer}
 \varphi_z(t):=\frac\mu2(t-W+z)^2
  +\frac{L-\mu}{2}\bigl(t-W+v(z)\bigr)_+^2.
\end{equation}
Since $v(z)\le z$, this function belongs to
$\mathcal F_{\mu,L}(\R)$, has minimum $0$ at $W-z$, and satisfies
\[
 \varphi_z(W)=d,\qquad
 \varphi'_z(W)=\mu z+(L-\mu)v(z).
\]
Moreover $\varphi_z(0)\le d$: when $W-z\le0$, the function is
nondecreasing on $[0,W]$; otherwise, $0<W-z\le R_d$ and
$\varphi_z(0)=\mu(W-z)^2/2\le d$.
By convexity, $\varphi_z\le d$ throughout $[0,W]$.
As $z$ decreases continuously from $R_d$ to $z_0$, its terminal slope
increases continuously from $\sqrt{2\mu d}$ to $A_d(W)$.
Indeed, $z_0=S_d$ when $W\le R_d+S_d$, and $z_0=W-R_d$
otherwise; substitution gives the two branches of \eqref{eq:Ad}.
This proves the full asserted range of attainable slopes.
\end{proof}

For \eqref{eq:intermediate}, define
\[
p:=\varepsilon-c,
\qquad
r:=\sqrt{\frac{2c}{L}},
\qquad
R_c:=\sqrt{\frac{2c}{\mu}},
\qquad
R_d:=\sqrt{\frac{2d}{\mu}},
\]
and
\begin{equation}
\label{eq:Id}
I_d:=
\left[
 r,
 \min\{R_c,2R_d-r\}
\right].
\end{equation}
For $a\in I_d$, set
\begin{equation}
\label{eq:H}
H(a):=
\frac{2p}{\sqrt{\beta(a)^2+2\mu p}+\beta(a)},
\end{equation}
and
\begin{equation}
\label{eq:Phi}
\Phi_d(a):=
A_d(r+a)H(a)+\frac L2H(a)^2.
\end{equation}

\begin{theorem}[Exact fixed-curvature envelope]
\label{thm:envelope}
Let $0<\mu\le L$, $c\ge0$, and $\varepsilon\ge0$. The fixed-curvature
information class is nonempty if and only if \eqref{eq:feascurv} holds.
Whenever it is nonempty,
\begin{equation}
\label{eq:envelope}
\mathfrak E_{\mu,L}
=
\begin{cases}
d,
& c\ge\varepsilon,\\[1mm]
d+\varepsilon-c,
& 0\le c<\varepsilon,\ d\ge c,\\[2mm]
\displaystyle
d+
\min\left\{
\varepsilon-c,
\max_{a\in I_d}\Phi_d(a)
\right\},
& 0<c<\varepsilon,\ \rho c\le d<c.
\end{cases}
\end{equation}
The third branch occurs only when $\mu<L$. Every finite value in
\eqref{eq:envelope} is attained in dimension one by piecewise-quadratic
objectives in $\mathcal F_{\mu,L}(\R)$.
\end{theorem}

\begin{proof}
The feasibility condition is Theorem~\ref{thm:barrier}. We first prove the
upper bound in the intermediate regime \eqref{eq:intermediate}. Normalize
\[
f:=F-F^\star,
\qquad
h:=\widetilde F-m.
\]
The certificate condition $\delta_{+,\varepsilon}\le\delta$ is equivalent to
\begin{equation}
\label{eq:obstacle}
f(x)\le d+(h(x)-c)_+
\qquad\text{whenever }h(x)\le\varepsilon.
\end{equation}
If a selected point satisfies $h\le c$, then $f\le d$, so consider a selected
$\widehat x$ with $h(\widehat x)>c$. Restrict both functions to the line from
the minimizer of $h$ to $\widehat x$, put that minimizer at coordinate $0$,
and let $a>0$ be the first positive coordinate at which $h(a)=c$.

Set $r=\sqrt{2c/L}$. By $L$-smoothness, $h(-r)\le c$, and convexity gives
$h\le c$ throughout $[-r,a]$. Hence \eqref{eq:obstacle} yields
\[
f\le d\qquad\text{on }[-r,a].
\]
The endpoints of this interval are both within distance $R_d$ of the target
minimizer, so $r+a\le2R_d$. Also $h(a)=c$ implies
$r\le a\le R_c$. Therefore $a\in I_d$.

Lemma~\ref{lem:beta} gives $h'(a)\ge\beta(a)$. If $s$ is the distance from
$a$ to $\widehat x$, strong convexity of $h$ gives
\[
p\ge\beta(a)s+\frac\mu2s^2,
\]
so $s\le H(a)$. Lemma~\ref{lem:Ad} gives
$f'(a)\le A_d(r+a)$, and $L$-smoothness therefore yields
\[
f(\widehat x)
\le d+A_d(r+a)s+\frac L2s^2
\le d+\Phi_d(a).
\]
The obstacle itself also gives $f(\widehat x)\le d+p$. Taking the supremum
over $a\in I_d$ proves the upper bound in the third branch.

We now prove sharpness in that regime. Fix $a\in I_d$ and
choose the surrogate $h=h_a$ from that equation. It satisfies
$h(0)=0$, $h(a)=c$, $h'(a)=\beta(a)$, and its negative-side
curvature is $L$. Its positive selected endpoint is $a+H(a)$.

Set $W=r+a$ and $A=A_d(W)$, $\beta=\beta(a)$, $H=H(a)$.
Choose a terminal target slope
\[
 u:=\min\{A,\beta\}.
\]
By \eqref{eq:beta-range} and $d<c$,
$\beta\ge\sqrt{2\mu c}>\sqrt{2\mu d}$.
Thus $u\in[\sqrt{2\mu d},A]$, and the construction
\eqref{eq:Ad-extremizer} supplies a global function
$f_{\rm left}(x)=\varphi_z(x+r)$ for a suitable $z$, such that
\[
 f_{\rm left}(a)=d,\quad f'_{\rm left}(a)=u,
 \quad 0\le f_{\rm left}\le d\ \text{on }[-r,a].
\]
We keep $f=f_{\rm left}$ on $(-\infty,a]$ and continue it to the
right with constant curvature $\alpha\in[\mu,L]$:
\[
 f(a+t)=d+ut+\frac\alpha2t^2\qquad(t\ge0).
\]
The functions join continuously with matching first derivatives;
thus $f$ is globally $\mu$-strongly convex and $L$-smooth.
Its global minimum is still $0$, attained to the left of $a$.

If $A<\beta$ and $\Phi_d(a)\le p$, take $u=A$ and
$\alpha=L$. Then $f(a+H)=d+\Phi_d(a)$.
For $0\le t\le H$, the difference
\[
 Q(t):=f(a+t)-h(a+t)-(d-c)
\]
is convex, starts at $Q(0)=0$ with $Q'(0)=A-\beta<0$,
and ends at $Q(H)=\Phi_d(a)-p\le0$. Convexity therefore gives
$Q(t)\le0$ throughout the selected positive band.

If $A<\beta$ and $\Phi_d(a)>p$, choose
\[
 u=A,\qquad
 \alpha=\mu+\frac{2(\beta-A)}{H}\in(\mu,L).
\]
The upper restriction on $\alpha$ follows precisely from
$\Phi_d(a)>p=\beta H+\mu H^2/2$. In this case
\[
 Q(t)=-(\beta-A)t\left(1-\frac{t}{H}\right)\le0
 \qquad(0\le t\le H),
\]
and $f(a+H)=d+p$.
Finally, if $A\ge\beta$, take $u=\beta$ and $\alpha=\mu$;
then $Q(t)\equiv0$ on $[0,H]$ and again $f(a+H)=d+p$.
Consequently, in every case
\[
 f(a+H)=d+\min\{p,\Phi_d(a)\},
\]
with the obstacle valid on the positive selected band.

To check the rest of the selected set, note first that
$h\le c$ on $[-r,a]$ and $f\le d$ there. For $x\le-r$,
we have $h''=L$ and $f''\le L$ almost everywhere. Furthermore,
$h(-r)=c$, $h'(-r)=-Lr=-\sqrt{2Lc}$, and
$0\le f(-r)\le d$. Smooth convexity gives
\[
 f'(-r)\ge-\sqrt{2Lf(-r)}\ge-\sqrt{2Ld}
 >-\sqrt{2Lc}=h'(-r).
\]
Hence $Q_-(x):=f(x)-h(x)-(d-c)$ satisfies
$Q_-(-r)\le0$, $Q_-''\le0$, and $Q_-'(-r)>0$.
It follows that $Q_-(x)\le Q_-(-r)\le0$ for every $x\le-r$.
This also verifies the negative selected band.
The constructed functions therefore satisfy the constraint
\eqref{eq:obstacle} everywhere on $\{h\le\varepsilon\}$.
Taking the maximizing $a\in I_d$, which exists by compactness,
proves sharpness in the intermediate regime.

The first branch has a direct quadratic extremizer. Put
\[
s:=\sqrt{\frac{2\varepsilon}{L}},
\qquad
R:=\sqrt{\frac{2d}{\mu}},
\qquad
h(x)=\frac L2x^2,
\qquad
f(x)=\frac\mu2(x-(s-R))^2.
\]
Feasibility gives $R\ge s$. The selected set is $[-s,s]$, it lies entirely
below the threshold when $c\ge\varepsilon$, and $f\le d$ on it, with
$f(s)=d$. Hence the envelope equals $d$.

For the second branch, let
\[
S:=\sqrt{\frac{2\varepsilon}{\mu}},
\qquad
C:=\sqrt{\frac{2c}{\mu}},
\qquad
E:=d+\varepsilon-c,
\qquad
R:=\sqrt{\frac{2E}{\mu}},
\]
and take
\[
h(x)=\frac\mu2x^2,
\qquad
f(x)=\frac\mu2(x-(S-R))^2.
\]
Since $d\ge c$, one has $S-R\le0$ and $f(S)=E$. On the band $h\ge c$,
$f-h$ is affine and increasing, with
\[
(f-h)(S)=d-c,
\]
so \eqref{eq:obstacle} holds there. On $|x|\le C$, the maximum of $f$ is at
$x=C$, and
\[
f(C)
=
\left(\sqrt c-\sqrt\varepsilon+
\sqrt{d+\varepsilon-c}\right)^2
\le d.
\]
Indeed,
\[
\sqrt{d+\varepsilon-c}-\sqrt d
\le
\sqrt\varepsilon-\sqrt c
\]
follows from $d\ge c$. Hence the active part of \eqref{eq:obstacle} also
holds, and the value $d+\varepsilon-c$ is attained at $x=S$.
The constructions also prove sufficiency of the feasibility condition
in all three regimes. This completes the proof.
\end{proof}

\begin{remark}[A numerical illustration]\label{rem:example}
Take
\[
\mu=1,\qquad L=4,\qquad c=1,\qquad \varepsilon=\frac32,
\qquad d=\frac14.
\]
Here $d=(\mu/L)c$, so the unrestricted sharp bound is
\[
d+\varepsilon-c=\frac34.
\]
Moreover $I_d$ is the singleton
\[
I_d=\left\{\frac1{\sqrt2}\right\}.
\]
At $a=1/\sqrt2$,
\[
\beta(a)=2\sqrt2,
\qquad
H(a)=3-2\sqrt2,
\qquad
A_d(2a)=\frac1{\sqrt2},
\]
and therefore
\[
\mathfrak E_{1,4}
=\frac14+\Phi_d(a)
\approx 0.4302.
\]
Thus fixed curvature lowers the exact worst-case target loss from $0.75$ to
approximately $0.4302$, a reduction of about $43\%$ relative to the
unrestricted sharp bound.
\end{remark}

\section{When does fixed curvature improve the certificate?}
\label{sec:phase}

The intermediate formula in Theorem~\ref{thm:envelope} does not imply
strict improvement throughout the intermediate regime: the cap
$d+\varepsilon-c$ can remain active. The following criterion locates the
transition exactly, without a second optimization over the accuracy
parameter $\varepsilon$.

Assume $0<\mu<L$, $0<c<\varepsilon$, and
$\rho c\le d<c$. Retain $I_d$, $r$, $\beta$ and $A_d$ from
Section~\ref{sec:envelope}. For $a\in I_d$ define
\[
q_d(a):=\bigl[\beta(a)-A_d(r+a)\bigr]_+,
\qquad
\Psi_d(a):=\frac{2\beta(a)q_d(a)}{L-\mu}
             +\frac{2\mu q_d(a)^2}{(L-\mu)^2},
\]
and set
\begin{equation}\label{eq:pcrit}
p_{\rm crit}:=\min_{a\in I_d}\Psi_d(a).
\end{equation}
This is a well-defined, nonnegative minimum on a compact interval.

\begin{corollary}[Exact strict-improvement threshold]\label{cor:phase}
Under the assumptions above, with $p=\varepsilon-c>0$,
\[
\mathfrak E_{\mu,L}<d+p
\quad\Longleftrightarrow\quad
0<p<p_{\rm crit}.
\]
For $p\ge p_{\rm crit}$ the unrestricted sharp value $d+p$ is recovered.
At the feasibility boundary $d=\rho c$ the threshold simplifies to
\begin{equation}\label{eq:boundarycrit}
p_{\rm crit}=4c(1+\rho).
\end{equation}
\end{corollary}

\begin{proof}
For $a\in I_d$, put $H=H(a)>0$, $A=A_d(r+a)$ and
$\beta=\beta(a)$. By definition,
\[
p=\beta H+\frac\mu2 H^2,
\qquad
\Phi_d(a)-p=(A-\beta)H+\frac{L-\mu}{2}H^2.
\]
If $A\ge\beta$, then $\Phi_d(a)\ge p$ for every $p>0$, and
$\Psi_d(a)=0$. Otherwise $\Phi_d(a)\ge p$ precisely when
$H\ge 2(\beta-A)/(L-\mu)$, equivalently
$p\ge\Psi_d(a)$ since $p=\beta H+\mu H^2/2$ is strictly increasing
in $H\ge0$. Maximizing over the compact set $I_d$ therefore yields
\[
\max_{a\in I_d}\Phi_d(a)<p
\quad\Longleftrightarrow\quad
p<\min_{a\in I_d}\Psi_d(a).
\]
The formula for $\mathfrak E_{\mu,L}$ gives the equivalence.
At $d=\rho c$, one has $I_d=\{r\}$,
$\beta(r)=Lr$, $A_d(2r)=\mu r$, and $r^2=2c/L$.
Substitution in \eqref{eq:pcrit} gives \eqref{eq:boundarycrit}.
\end{proof}

The example in Remark~\ref{rem:example} has $\rho=1/4$,
$c=1$, and $p_{\rm crit}=5$. Thus the improvement is strict for
$0<\varepsilon-c<5$, but not for all accuracy tolerances.
For the same $(\mu,L,c,\varepsilon)=(1,4,1,3/2)$, increasing $d$
from $1/4$ to $7/20$ removes the strict gap: the envelope then equals
$d+\varepsilon-c=0.85$.

\section{Localized and finite validation}\label{sec:validation}

The curvature envelope uses a known crossing depth $c=\ell-m$ and a
certified selected-set discrepancy. By contrast, the general bound
needs only an upper bound on $m$.

\subsection{Level bounds and an unknown target optimum}
Exact knowledge of $m$ is unnecessary for the general certificate. If
$m\le\overline m$, then Theorem~\ref{thm:master} gives
\begin{equation}\label{eq:upperlevel}
F(\widehat x)-F^\star
\le\delta_{+,\varepsilon}
+(\overline m+\varepsilon-\ell)_+-g.
\end{equation}
A feasible surrogate witness $x_{\rm w}$ immediately supplies
$\overline m=\widetilde F(x_{\rm w})$. It also yields the
certified outer set
$\A_\varepsilon(\widetilde F)\subset
\{x:\widetilde F(x)\le\overline m+\varepsilon\}$,
which can be used for local validation. The witness does not, by itself,
certify $\delta_{+,\varepsilon}$; that requires a separate comparison.
Importantly, an upper bound on $m$ alone does not determine the exact
crossing depth $c=\ell-m$ required by the fixed-curvature envelope.

For the fixed-curvature envelope, the target-optimum offset
$g=F^\star-\ell$ need not be known, provided the surrogate minimum $m$ is
known. The resulting bound is conservative, since it discards the possible
positive offset $g$.

\begin{corollary}[Fixed-curvature certification without $F^\star$]
\label{cor:unknown-optimum}
Let $F,\widetilde F\in\mathcal F_{\mu,L}(\R^n)$.
Assume $\ell\le F^\star$ and $m=\inf\widetilde F=\ell-c$
for some $c\ge0$. Suppose further that
$\delta_{+,\varepsilon}\le\overline\delta$. Then every
$\widehat x\in\A_\varepsilon(\widetilde F)$ satisfies
\begin{equation}\label{eq:gfree}
 F(\widehat x)-F^\star
 \le \mathfrak E_{\mu,L}(\ell,0,c,\varepsilon,\overline\delta).
\end{equation}
The envelope on the right is the explicit expression in
Theorem~\ref{thm:envelope} with $g=0$ and $d=\overline\delta$.
\end{corollary}

\begin{proof}
Let $g=F^\star-\ell\ge0$ and define $F_0=F-g$. Then
$F_0\in\mathcal F_{\mu,L}(\R^n)$, $\inf F_0=\ell$, and
$F_0(\widehat x)-\inf F_0=F(\widehat x)-F^\star$. Since $g\ge0$,
\[
 \sup_{x\in\A_\varepsilon(\widetilde F)}
 [F_0(x)-T_\ell(\widetilde F(x))]_+
 \le\delta_{+,\varepsilon}\le\overline\delta.
\]
Apply Theorem~\ref{thm:envelope} to $(F_0,\widetilde F)$.
\end{proof}

\subsection{Threshold-preserving local validation}

If $F(x)\le\widetilde F(x)+\eta$ on a set containing
$\A_\varepsilon(\widetilde F)$, then
$\delta_{+,\varepsilon}\le\eta$, so the general bound holds with
$\delta_{+,\varepsilon}$ replaced by $\eta$. However, when
$m=\ell-c$, evaluation at the surrogate minimizer gives
\[
 \eta\ge F(\widetilde x^\star)-\widetilde F(\widetilde x^\star)
 \ge g+c.
\]
Thus this \emph{unthresholded} validation cannot reach the
curvature-improvement regime $\eta-g<c$. To preserve it, we evaluate
the thresholded discrepancy directly.

\begin{proposition}[Threshold-preserving finite-net validation]
\label{prop:finite-net}
Suppose $B$ is a subset of a normed space and
$S\supset\A_\varepsilon(\widetilde F)$ is totally bounded.
Suppose $F$ and $\widetilde F$ are Lipschitz on $S$, with constants
$K_F$ and $K_{\widetilde F}$. Let
$\{x_1,\ldots,x_N\}\subset S$ be an $r$-net of $S$, and suppose
certified upper values $u_i\ge F(x_i)$ are available. Put
\begin{equation}\label{eq:threshold-net}
\overline\delta_r:=
\max_{1\le i\le N}
\bigl[u_i-T_\ell(\widetilde F(x_i))\bigr]_+
+(K_F+K_{\widetilde F})r.
\end{equation}
Then $\delta_{+,\varepsilon}\le\overline\delta_r$.
In particular, the general bound \eqref{eq:upperlevel} holds with
$\delta_{+,\varepsilon}$ replaced by $\overline\delta_r$.
When $F,\widetilde F\in\mathcal F_{\mu,L}(\R^n)$ and the surrogate
minimum $m=\ell-c\le\ell$ is known, the bound
\eqref{eq:gfree} also holds with
$\overline\delta=\overline\delta_r$.
\end{proposition}

\begin{proof}
The maps $z\mapsto[z]_+$ and $T_\ell$ are $1$-Lipschitz.
Thus $G(x):=[F(x)-T_\ell(\widetilde F(x))]_+$ is
$(K_F+K_{\widetilde F})$-Lipschitz on $S$.
Given $x\in\A_\varepsilon(\widetilde F)$, choose $x_i$ with
$\|x-x_i\|\le r$. Then
\[
G(x)\le G(x_i)+(K_F+K_{\widetilde F})r
\le\overline\delta_r.
\]
Taking the supremum over the selected set proves the first claim.
The remaining assertions follow from Theorem~\ref{thm:master} and
Corollary~\ref{cor:unknown-optimum}.
\end{proof}

The distinction between \eqref{eq:threshold-net} and an unthresholded
net estimate is essential for the intermediate regime: the positive part
and the threshold are retained at every validation point. Covering numbers
can nevertheless be large in high dimensions. The result supplies a
sufficient finite verification procedure, not a dimension-free algorithm.

\begin{remark}[Finite validation in the strict-improvement regime]
\label{rem:finite-example}
Let $F(x)=\tfrac12(x-1/50)^2$ and $\widetilde F(x)=2x^2-1$,
with $(\mu,L,\ell,c,\varepsilon)=(1,4,0,1,13/10)$.
Here $F^\star=0$ and
$\A_\varepsilon(\widetilde F)=[-\sqrt{13/20},\sqrt{13/20}]$.
For a fully rational finite validation, cover this selected set by
$S=[-807/1000,807/1000]$ and take the $1001$ equally spaced nodes
$x_i=-807/1000+i(807/500000)$, $0\le i\le1000$.
The covering radius is $r_{\rm net}=807/10^6$; on $S$ the Lipschitz
constants $K_F=827/1000$ and $K_{\widetilde F}=807/250$ are valid.
With exact rational nodal evaluations, \eqref{eq:threshold-net} gives
\[
 \overline\delta_r=\frac{16717965707}{62500000000}
 <0.268<0.27.
\]
Taking $\overline\delta=0.27$, the general certificate is $0.57$.
The fixed-curvature formula yields numerically
$\mathfrak E_{1,4}(0,0,1,1.3,0.27)\approx0.44054$.
To certify a strict improvement without relying on this numerical
maximizer, write $I_d=[a_0,a_5]$, where
\[
\begin{aligned}
 a_0&=1/\sqrt2, &a_1&=18/25, &a_2&=73/100,\\
 a_3&=37/50, &a_4&=3/4, &a_5&=2\sqrt{27/50}-1/\sqrt2.
\end{aligned}
\]
On $I_d$, $H(a)$ is nondecreasing and $A_d(a_0+a)$ is nonincreasing.
Therefore
\[
 \max_{a\in I_d}\Phi_d(a)
 \le\max_{0\le j<5}
 \left\{A_d(a_0+a_j)H(a_{j+1})+2H(a_{j+1})^2\right\}
 <0.178.
\]
The final strict inequality can be verified by outward rational
square-root bounds; an exact-arithmetic checking script accompanies
this example. Consequently
\[
 F(\widehat x)-F^\star <0.27+0.178=0.448<0.45,
\]
whereas the general information bound of $0.57$ does not certify
an accuracy of $0.45$. This illustrates the benefit of the reduced
information envelope, not computational superiority over methods
that exploit the target gradient or directly solve this quadratic.
\end{remark}

\begin{remark}[Unbounded ambient discrepancy]
Let $B=[-1,1]$, $F(x)=x^2$, $\ell=0$, $r_0=\sqrt\delta/2$ for
$0<\delta<1$, and
$\widetilde F_K(x)=K(x-r_0)^2-\varepsilon$, where
$K\ge4\varepsilon/\delta$. Then $m=-\varepsilon$ and
$\A_\varepsilon(\widetilde F_K)\subset[0,\sqrt\delta]$.
Thus $\delta_{+,\varepsilon}\le\delta$ and the general
certificate is at most $\delta$, while
$\|F-\widetilde F_K\|_{\infty,B}\to\infty$ as $K\to\infty$.
Localization can discard arbitrarily large ambient discrepancies;
this does not imply that local validation is always cheaper.
\end{remark}

\section{Discussion and conclusion}

The analysis identifies two distinct components of post-selection
information. The one-sided discrepancy controls the target objective on
the surrogate-selected set, whereas a bound on the surrogate minimum
controls its absolute level. Together they yield the sharp general
certificate of Theorem~\ref{thm:master}; neither a global two-sided
comparison nor an algorithm-specific argument is required.

With fixed strong convexity and smoothness, the problem admits a sharper
characterization. The feasibility barrier specifies which information
budgets are compatible with the curvature class, and
Theorem~\ref{thm:envelope} determines the exact worst-case target loss
in every feasible regime. Its one-dimensional constructions attain the
bound under the \emph{whole} selected-set constraint.
Corollary~\ref{cor:phase} locates precisely the range in which curvature
strictly improves the general certificate. Hence the envelope also
identifies when no stronger uniform bound can follow from the same data.

The threshold-preserving finite-net estimate provides one route to
certifying the selected discrepancy from finitely many evaluations.
The illustrative quadratic pair yields a general bound of $0.57$ and
a fixed-curvature bound below $0.448$, showing a strict gain in
certification strength. Applying the exact envelope requires the
surrogate minimum and a certified discrepancy budget; obtaining these
quantities efficiently is a separate numerical question.
The envelope also resolves a precise sufficiency question: for any
prescribed accuracy tolerance, it determines whether the retained
information certifies all admissible surrogate near-minimizers.
Whenever the envelope exceeds that tolerance, the attaining examples
show that no stronger \emph{uniform} guarantee follows from the same
information. Additional observations or structural assumptions are then
necessary to improve the certification.

\section*{Statements and Declarations}

\noindent\textbf{Funding.}
No funds, grants, or other support were received.

\medskip
\noindent\textbf{Competing interests.}
The authors have no relevant financial or non-financial interests to disclose.

\medskip
\noindent\textbf{Data availability.}
No datasets were generated or analysed during the current study.

\end{document}